\documentclass[reqno,10pt]{amsart}

\usepackage[english]{babel}
\usepackage{dsfont}
\usepackage[bookmarksnumbered,colorlinks]{hyperref}
\usepackage{enumerate}
\newcommand{\df}{{\rm d}}
\newcommand{\R}{\mathds R}
\newcommand{\N}{\mathds N}

\title[Conformally standard stationary spacetimes and Fermat metrics]{Conformally standard 
stationary spacetimes and Fermat metrics}

\author[M. A. Javaloyes]{Miguel Angel Javaloyes}
\address{Departamento de Matem\'aticas, \hfill\break\indent
Universidad de Murcia, \hfill\break\indent
Campus de Espinardo,\hfill\break\indent
30100 Espinardo, Murcia, Spain}
\email{majava@um.es}

\date{15.11.2011}

\thanks{The author is partially supported by
Regional J. Andaluc\'{\i}a Grant P09-FQM-4496, by  MICINN project MTM2009-10418 and
Fundaci\'on S\'eneca project 04540/GERM/06.}

\subjclass[2010]{}


\begin{document}

\newtheorem{thm}{Theorem}[section]
\newtheorem{prop}[thm]{Proposition}
\newtheorem{lemma}[thm]{Lemma}
\newtheorem{cor}[thm]{Corollary}
\theoremstyle{definition}
\newtheorem{defi}[thm]{Definition}
\newtheorem{notation}[thm]{Notation}
\newtheorem{exe}[thm]{Example}
\newtheorem{conj}[thm]{Conjecture}
\newtheorem{prob}[thm]{Problem}
\newtheorem{rem}[thm]{Remark}
\begin{abstract}
In this review, we collect several results for conformally standard stationary spacetimes $(S\times\R,g)$
obtained in terms of a Finsler metric of Randers type on the orbit manifold $S$ that we call Fermat metric. This metric is obtained by applying the relativistic Fermat principle and it turns out that 
it encodes all the causal aspects of the spacetime.
\end{abstract}
\maketitle

\section{Introduction}
Fermat's principle, say, that light rays minimize the arrival time, is linked to General Relativity from its very beginning. As early as 1917, H. Weyl established a version for static spacetimes in \cite{weyl17} and several other authors, as T. Levi-Civita and
J. L. Synge \cite{LeCi18,syn25}, gave some attention to the principle. Not much later, in 1927, T. Levi-Civita stated the stationary version in \cite{LeCi27} (see also \cite{Pham57}), and 
it was included in the book \cite{LauLif62}. The general version was formulated by I. Kovner in 1990 \cite{Kov90} and rigorous established by V. Perlick in \cite{Per90} (see also \cite{Per08} for a version in Finsler spacetimes). 

Independently from Fermat's principle, Randers metrics appeared as an attempt of G. Randers to geometrize electromagnetism in General Relativity \cite{Ran41}, but it seems that it was R. Ingarden the first one that thought
in Randers metrics as Finsler ones in his PhD thesis \cite{Ing57}. By the way, R. Miron \cite{Mir04} suggested to name  the Randers metrics endowed with a non-linear Lorentz connection (associated to the Lorentz equation in electrodynamics) as {\em Ingarden spaces}.  Afterwards they were recovered by
M. Matsumoto with the aim of giving examples of the so-called C-reducible Finsler metrics. In order to obtain these examples, he introduced the class of $(\alpha,\beta)$-metrics in a manifold $M$, that is, Finsler metrics that are obtained as a homogeneous combination of the square root of a Riemannian metric $h$  and a one-form $\beta$ on $M$  (with the notation $\alpha(v)=\sqrt{h(v,v)}$ for $v\in TM$) \cite{Mat72}. In particular, Randers metrics are defined as
$\alpha+\beta$.
This function is positively homogeneous but not reversible. Moreover, it is positive whenever the $h$-norm of $\beta$ is less than $1$ in every point.
Subsequently, the Japanese school of Finsler Geometry spent some time studying Randers metrics, mostly problems related with curvature \cite{Mat74,SSAY77,YaSh77}. Let us point out that the approach of G. Randers himself was somewhat different, since he constructed his metric from a Lorentzian metric and a one-form in the spacetime. It is also remarkable that A. Lichnerowicz and Y. Thiry obtained a Randers
metric when studying Jacobi-Maupertuis principle in General Relativity (see \cite{LiTh47} and \cite[pg. 155]{Li55}).

In this review, we will describe some recent results that use techniques of Finsler geometry to study
conformally standard stationary spacetimes and vice versa. 
\section{Finsler and Randers metrics\label{FinslerMetrics}} There are several definitions of Finsler metrics \cite{JaSan11}. 
But the most general case where you can extend most of the classic Riemannian results is the following. 
Let $\pi:TM\rightarrow M$ be the natural projection from the tangent bundle to the manifold.
A {\em Finsler metric} is a continuous function $F:TM\rightarrow [0,+\infty)$ satisfying the following properties:
\begin{enumerate}
\item $F$ is $C^\infty$ in $TM\setminus 0$, i. e. away from the zero section,
\item $F$ is fiberwise positively homogeneous of degree one, i. e. $F(\lambda v)=\lambda F(v)$ for every $v\in TM$ and $\lambda>0$,
\item $F^2$ is fiberwise strongly convex, i.e., the fundamental tensor $g_u$ defined as
\begin{equation}\label{fundamentaltensor}
g_u(v,w)=\frac{\partial^2}{\partial s\partial t}F^2(u+tv+sw)|_{t,s=0},
\end{equation}
where $u\in TM\setminus 0$ and $v,w\in T_{\pi(u)}M$, is positively defined for every $u\in TM\setminus 0$.
\end{enumerate}
These conditions imply that $F$ is positive away from the zero section,  
the triangle inequality holds for $F$ in the fibers (see \cite[Section 1.2B]{BaChSh00})
and $F^2$ is $C^1$ \cite{Warner65}. Property $(3)$ above is
essential to guarantee minimization properties of geodesics. The first geometers that worked with Randers metrics seemed very concerned with computation of curvatures and invariants related with connections and, apparently, they
overlooked the question of strong convexity. Let us recall that a Randers metric on a manifold
$M$ is constructed using a Riemannian metric $h$ and a one-form $\beta$ on $M$ as
\begin{equation}
\label{Randers}
R(v)=\sqrt{h(v,v)}+\beta(v)
\end{equation}
for every $v\in TM$. It turns out that it
is fiberwise strongly convex if and only if it is positive for every $v\in TM$. 
This can be easily seen
computing the fundamental tensor (see \cite[Corollary 4.17]{JaSan11}):
\[g_v(w,w)=\frac{R(v)}{\sqrt{h(v,v)}}h(w,w)+
\left(\frac{h(v,w)}{\sqrt{h(v,v)}}+\beta(w)\right)^2,\]
with $v\in TM\setminus 0$ and $w\in T_{\pi(v)}M$.
Up to our knowledge, the first time that a proof of this fact appeared was in \cite[Section 11.1]{BaChSh00}
published in 2000.

Positive homogeneity of Finsler metrics implies that the length of a piecewise smooth curve $\gamma:[a,b]\subseteq \R
\rightarrow M$ given by
\[\ell_F(\gamma)=\int_a^b F(\dot\gamma) ds\]
 does not depend on the orientation preserving parametrization of the curve. Then you can define the distance between two points $p,q\in M$ as
 \[d(p,q)=\inf_{\gamma\in C_{p,q}}\ell_F(\gamma),\]
where $C_{p,q}$ is the space of piecewise smooth curves from $p$ to $q$. This gives a generalized distance (see \cite[pg. 5]{Zau} and also \cite{FHS10b,JaLiPi11}), but not necessarily reversible as the length of a curve
 depends on the orientation of the parametrization (observe that in general $F(-v)\not=F(v)$). Then
 you can define two kind of balls, that is, forward and backward balls
, respectively, as
\[B_F^+(p,r)=\{q\in M: d_F(p,q)<r\},\quad B_F^-(p,r)=\{q\in M: d_F(q,p)<r\},\]
for every $p\in M$ and $r>0$. Moreover, there exist several definitions for Cauchy sequences.
\begin{defi}\label{Cauchy}
 A sequence $\{x_n\}_{n\in\N}$ is called a {\em forward (resp. backward) Cauchy sequence} if
 for any $\varepsilon>0$, there exists
 $N\in\N$ such that $d_F(x_i,x_j)<\varepsilon$ for any $i,j\in\N$ satisfying $N<i<j$ (resp. $N<j<i$).
\end{defi}
 Moreover, you can also define the energy functional as
 \[E_F(\gamma)=\frac 12\int_a^b F(\dot\gamma)^2 ds\]
for every piecewise smooth curve $\gamma:[a,b]\subseteq \R
\rightarrow M$,  and geodesics as critical points of this functional. In particular, geodesics must have constant speed
 (see for example \cite[Proposition 2.1]{CJM11}). Let us point out that in some references as \cite{BaChSh00} geodesics are defined as critical points of the length functional and as a consequence
 they are not assumed to have constant speed.
\section{Fermat's principle in conformally standard stationary spacetimes}
Let us recall that a conformally stationary spacetime is a Lorentz manifold $(M,g)$ that admits a timelike conformal vector field $K$. We refer to the classical books \cite{BeEr,O'N83} for the basic definitions on Lorentzian Geometry and Causality. Observe that $K$ determines a time-orientation in $(M,g)$ and thus, a spacetime 
that, with an abuse of notation, we will denote also by $(M,g)$. It can be shown that when $K$ is complete and the spacetime is distinguishing (see \cite{JaSan08}), then $(M,g)$ splits as a {\em 
conformally standard stationary spacetime}, that is, $M=S\times \R$ and the metric $g$ can be expressed as
\begin{equation}\label{e1conf} g((v,\tau),(v,\tau))=\varphi(g_0(v,v)+2 \omega(v)\tau-\tau^2),
\end{equation}
in $(x,t)\in S\times\R$, where $(v,\tau)\in T_xS\times \R$, $\varphi$ is a smooth positive function on $S\times\R$ and  $\omega$ and $g_0$ are respectively a one-form and a Riemannian metric 
on the manifold $S$. In this setting, the vector field $K$ is induced from the natural
lifting to $M$ of the canonical vector field $ d/dt$ on $\R$, which we will denote
in the following by $\partial_t$. Let us remark that sometimes in literature the one-form $\omega$ is expressed using the metrically equivalent vector field $\delta$, that is,
$\omega(v)=g_0(v,\delta)$ for every $v\in TS$. 

We must observe that, in a Lorentzian manifold, lightlike geodesics and their conjugate points 
are preserved 
by conformal changes up to parametrization (see for example 
\cite[Theorem 2.36]{MinSan08}). This implies that studying lightlike geodesics of
$(S\times \R,g)$ is equivalent to studying lightlike geodesics of $(S\times\R,\frac{1}{\varphi}g)$.
As a consequence we can assume that the spacetime is a {\em normalized standard stationary spacetime}, that is, a standard stationary spacetime with a unit Killing vector field  and
\begin{equation}\label{e1} g((v,\tau),(v,\tau))=g_0(v,v)+2 \omega(v)\tau-\tau^2,
\end{equation}
in $(x,t)\in S\times\R$ for any $(v,\tau)\in T_xS\times \R$. In this case, $\partial_t$ rather than a conformal vector field is a unit Killing
vector field. 

The advantage of  formulating the Fermat's principle in  (conformally) standard stationary spacetimes 
is that it is possible to define a global time function given by the second coordinate in $S\times \R$ and
it also makes sense to speak about the spatial position, that is, the first coordinate. Now fix two
spatial positions $x_0$ and $x_1$ in $S$. Then Fermat's principle says that lightlike geodesics are critical points of the global time function between all the possible trajectories for light rays from 
$(x_0,t_0)$ to $(x_1,t_1)$, with $t_0,t_1\in\R$ (here we can fix $t_0$ but not $t_1$). According to General Relativity, as photons are massless,  the trajectories of light rays must be described by lightlike curves. Therefore, the space of curves for the Fermat's principle must be composed of smooth future-pointing lightlike curves. Let us observe that as the time-orientation is assumed to be given by the Killing
vector field $\partial_t$, a future-pointing causal curve is a curve $\gamma=(x,t):[a,b]\subseteq\R\rightarrow S\times \R$ satisfying that $g(\dot\gamma,\dot\gamma)\leq 0$ and $\dot t>0$. If $\gamma=(x,t):[0,1]\rightarrow S\times \R$ is a smooth lightlike curve from $(x_0,t_0)$ to $(x_1,t_1)$, we need to compute $t_1$, which is the value of the global time in
$\gamma(1)$. As $\gamma$ is lightlike, we have that
\[g_0(\dot x,\dot x)+2\omega(\dot x) \dot t- \dot t^2=0,\]
and hence, as $\gamma$ is assumed to be future-pointing (that is, $\dot t>0$),
\[\dot t=\sqrt{g_0(\dot x,\dot x)+\omega(\dot x)^2}+\omega(\dot x).\]
Integrating last equation we get 
\begin{equation}
\label{tempo}
t(s)=t_0+\int_0^s\left(\sqrt{g_0(\dot x,\dot x)+\omega(\dot x)^2}+\omega(\dot x)\right) dv.
\end{equation}
As a consequence, lightlike geodesics must be critical points of the functional
\[T(\gamma)=t_1=t_0+\int_0^1\left(\sqrt{g_0(\dot x,\dot x)+\omega(\dot x)^2}+\omega(\dot x)\right) dv.\]
This functional is, up to a constant, the length functional of the Finsler metric in $S$ given by
\begin{equation}\label{fermatmetric}
F(v)=\sqrt{g_0(v,v)+\omega(v)^2}+\omega(v),
\end{equation}
for every $v\in TS$. This metric is of Randers type, that is, the addition of the square of a Riemannian metric and a one-form of norm less than one in every point. We will call this metric the {\em Fermat metric} associated to the splitting \eqref{e1} (or in general to the splitting \eqref{e1conf}). 
\begin{rem}
With a similar reasoning, we get that past-pointing lightlike geodesics are controlled by the reverse metric of \eqref{fermatmetric}, that is,
\[\tilde{F}(v)=F(-v)\]
for every $v\in TS$. It is easy to see that 
\begin{enumerate}
\item[(i)] $d_{\tilde{F}}(p,q)=d_{F}(q,p)$ for every $p,q\in S$,
\item[(ii)] $\gamma:[0,1]\subseteq\R\rightarrow S$ is a geodesic from $p$ to $q$ of
$(S,F)$ if and only if the reverse curve $\tilde{\gamma}:[0,1]\rightarrow S$, $t\to \tilde{\gamma}(t)=\gamma(1-t)$ is a geodesic from $q$ to $p$ of $(S,\tilde{F})$.
\end{enumerate}
Then all the properties of past-pointing lightlike geodesics can be also written in terms of the Fermat metric \eqref{fermatmetric}.
\end{rem}
\begin{rem} Let us consider the class of standard  static spacetimes $(S\times\R,g_{st})$, with 
\[g_{st}((v,\tau),(v,\tau))=g_0(v,v)-\beta(x) \tau^2,\]
in $(x,t)\in S\times\R$, where $(v,\tau)\in T_xS\times\R$, $g_0$ is a Riemannian metric on $S$ and $\beta$ a positive smooth function on $S$. In particular they are standard stationary and the Fermat metric associated to them is Riemannian. Indeed it is conformal to the metric induced by $g_{st}$ in $S$, that is, $\frac{1}{\beta}g_0$. This fact was already pointed out in \cite[pg. 343]{LeCi27}. Up to the name of Fermat metric, other authors have used another name for the same concept, for example, in \cite{DoKe78,GiPe78}, it is used {\em optical metric} and in \cite{ACL88}, {\em optical reference geometry}.
\end{rem}
\begin{rem}
In the stationary case, it must be clarified that our terminology is different from that of \cite{Per90ii}, where the name of {\em Fermat metric} is used for the Riemannian metric in $S$ given by
\begin{equation}\label{perfermatmetric}
h(v,v)=g_0(v,v)+\omega(v)^2,
\end{equation}
for $v\in TS$.
He also introduces the name of {\em Fermat one-form} for $\omega$. Observe that then
our Fermat metric is the addition of the Fermat one-form and the square root of Perlick's Fermat metric. But our Fermat metric contains all the information and in fact it allows one to 
recover the Fermat one-form and 
Perlick's Fermat metric as
\begin{align*}
h(v,v)&=\frac{1}{4}\left( F(v)+F(-v)\right)^2,& \omega(v)&=\frac{1}{2}\left(F(v)-F(-v)\right),
\end{align*}
for any $v\in TS$, where $F$ is given in \eqref{fermatmetric}.
\end{rem}
The above computations show that in (conformally) standard stationary spacetimes, Fermat's principle relates future-pointing lightlike geodesics of $(S\times \R,g)$  as in \eqref{e1} with geodesics of the Finsler manifold $(S,F)$ with $F$ given in \eqref{fermatmetric} up to reparametrizations. Let us state
the relation including parametrizations.
\begin{thm}[Fermat's principle]\label{fermatprinciple}
Let $(S\times \R,g)$ be a standard stationary spacetime as in \eqref{e1}.  A curve $\gamma=(x,t):[a,b]\subseteq\R\rightarrow S\times \R$ is a lightlike geodesic of $(S\times \R,g)$ if and only if $x$ is a geodesic for the Fermat metric $F$ in \eqref{fermatmetric} parametrized to have constant $h$-Riemannian speed ($h$ as in \eqref{perfermatmetric}) and 
\[t(s)=t(a)+\int_a^s F(\dot x)\df \nu,\]
for every $s\in [a,b]$.
\end{thm}
\begin{proof}
The equivalence can be easily obtained computing the critical points of the length functional
for $F$ with $h$-constant Riemannian speed using the Levi-Civita connection $\nabla$ of $g_0$ and then the ligthlike critical points of the energy functional of $g$ using again $\nabla$ (see for example \cite[Theorem 4.1]{CJM11}).
\end{proof}

\begin{rem}
Let us point out that V. Perlick \cite{Per90ii} considers a more general case than conformally standard stationary
spacetimes. Basically, he considers a conformally stationary spacetime $(M,g)$ 
where the  flow lines of the conformal vector field $K$ of $(M,g)$ have a structure of Hausdorff manifold $\hat{M}$ and the natural projection $\pi:M\rightarrow \hat{M}$ is a principal fiber bundle with structure group $\R$, with the action given by the flow of $K$. Observe that as the fiber is $\R$, there always exists a section of the bundle (see for example \cite[page 58]{KN}). But the existence 
of a spacelike section is not guaranteed. In fact, assuming that $K$ is complete, this happens if and only if the spacetime is distinguishing (see \cite{JaSan08}). Given a section $S$ of the fiber bundle, we can express the metric of $(M,g)$ as in \eqref{e1}, but with $g_0$ not necessarily spacelike. In this case, the global time given by the second coordinate is not necessarily a time function, that is, it does not have to be strictly increasing in causal curves. 
As a consequence, the Fermat metric obtained in \eqref{fermatmetric} can be non-positive
along some directions of the tangent space. In fact, it is not difficult to see that the Fermat metric \eqref{fermatmetric} is a Finsler metric
(with the definition given in Section \ref{FinslerMetrics}) if and only if the section $S$ is spacelike.
\end{rem} 
It can be helpful to restate the Fermat's principle as follows (see \cite[Proposition 4.1]{CJS11}).
\begin{prop}\label{ppp}
Let $z_0=(x_0,t_0)$, $L_{x_1}=\{(x_1,t): t\in \R\}$ be,
respectively, a point and a vertical line in a standard stationary
spacetime. Then $z_0$ can be joined with $L_{x_1}$ by means of a
future-pointing  lightlike pregeodesic
$t\mapsto\gamma(t)=(x_\gamma(t),t)$  starting at $z_0$ if and
only if $x_\gamma$ is a unit  speed geodesic of the Fermat metric
$F$  which joins $x_0$ with $x_1$. 
In this case,
\[t_1-t_0=\ell_F(x_\gamma|_{[t_0,t_1]}).\]
\end{prop}
Let us observe that Fermat metric depends on the spacelike section you choose to obtain the standard splitting (which in some references as \cite{GHWW09,Per90} is called the gauge choice). Last proposition can be used to obtain the relation between two Fermat metrics associated
to different splittings of
the same stationary spacetime (with a fixed timelike Killing vector field $K$).
If $(S\times \R,g)$ is one of the splittings (with $g$ as in \eqref{e1}), the other one is determined by a section given by a smooth function $f:S\rightarrow \R$ as $S_f=\{(x,f(x))\in S\times \R: x\in S\}$. Then you can define the map $\psi_f:S\times \R\rightarrow S\times \R$ given as $\psi(x,t)=(x,t+f(x))$ for every $(x,t)\in S\times \R$. Therefore the other splitting is expressed as $(S\times \R,g_f)$, where $g_f=\psi^*_f(g)$ (here $*$ denotes the pullback operation).
\begin{prop}\label{F+df}
With the above notation, the Fermat metric associated to the splitting $(S\times \R,g_f)$ is $F_f=F-df$, where $F$ is the Fermat metric associated to $(S\times \R,g)$ and $df$ is the differential of the smooth function $f$.
\end{prop}
\begin{proof}
Observe that given a curve $\gamma:[-\varepsilon,\varepsilon]\rightarrow S$, with $\varepsilon>0$,
\[\hat{F}(\dot \gamma(0))=\left.\frac{d}{ds}\right|_{s=0}\ell_{\hat{F}}(\gamma|_{[0,s]}),\]
for any Finsler metric $\hat{F}$.
 Moreover,  as a consequence of Proposition \ref{ppp}, 
\begin{equation}\label{lengthF+df}
\ell_{F_f}(\gamma)=\ell_F(\gamma)+f(\gamma(-\varepsilon))-f(\gamma(\varepsilon))=\int_{-\varepsilon}^\varepsilon( F(\dot\gamma)-df(\dot\gamma))ds.\end{equation}
Given $v\in TS$, consider $\gamma:[-\varepsilon,\varepsilon]\rightarrow S$ such that
$\dot\gamma(0)=v$. Then
\[ F_f(v)=\left.\frac{d}{ds}\right|_{s=0}\ell_{F_f}(\gamma|_{[0,s]}))=F(v)-df(v)\]
for any $v\in TS$.
\end{proof}
\begin{prop}
Given an arbitrary function $f:S\rightarrow \R$, the section $S_f$ of $S\times \R$ is spacelike if and 
only if $F(v)>df(v)$ for every $v\in TS$.
\end{prop}
\begin{proof}
See also \cite[Proposition 5.8]{CJS11}.
\end{proof}
Now we can establish the so-called Stationary to Randers correspondence (see \cite{CJS11}). Let us call
${\rm Stat}(S\times \R)$ the space of standard stationary spacetimes with normalized Killing vector field $\partial_t$ and ${\rm Rand}(S)$ the space of Randers metrics on $S$. Then one has the bijective map
\begin{equation}\label{eee}
{\rm Stat} (S\times \R)  \rightarrow  {\rm Rand}(S), \quad
\quad g \mapsto F_{g},  \end{equation} 
where  $F_g$ is determined as in \eqref{fermatmetric}
by the same stationary data pair $(g_0,\omega)$ which determines
$g$ in \eqref{e1}. Moreover, we can define in both sets equivalence relations as
$$\begin{array}{clll}
R\sim R'& \Longleftrightarrow & R-R'= df & \hbox{for some smooth function $f$ on $S$},\\
g\sim g ' & \Longleftrightarrow & g'=\psi_f^*g & \hbox{for some
change of the initial section $\psi_f$},
\end{array}$$
and consider the corresponding quotient sets  {\rm Rand}$(S)/\sim$, {\rm Stat}$(S\times \R)/\sim$. Proposition \ref{F+df} says that the bijection (\ref{eee})
induces a well-defined bijective map between the quotients
$$({\rm Stat}(S\times \R)/\sim ) \rightarrow {\rm
(Rand}(S)/\sim ).
$$
This relation constitutes a very important issue for Randers metrics, because the global invariants in
the spacetime must be translated in invariants for the entire class of Randers metrics that differ in 
the differential of a function.
\section{Causality and Fermat metrics}
As, by Proposition \ref{ppp}, geodesics of Fermat metrics contain all the information of lightlike geodesics up to reparametrization, it turns out that Fermat metrics can be used to describe the chronological future and past of a given point. As a consequence, we can charaterize the causal conditions of a standard stationary spacetime in terms of the Fermat metric. This relation was established in \cite{CJS11} with some previous partial results in \cite{CJM11}.  Recall that 
we say that two
events $p$ and $q$ in a spacetime are
chronologically related, and write $p\ll q$ (resp. strictly
causally related $p< q$)  if there exists a future-pointing
timelike (resp.
 causal) curve $\gamma$ from $p$ to $q$; $p$ is causally related to $q$ if either $p<q$ or $p=q$, denoted $p\leq q$.
Then the {\it chronological future}  (resp. {\it causal future}) of
$p\in M$ is defined as $I^+(p)=\{q\in M : p\ll q\}$ (resp.
$J^+(p)=\{q\in M : p\leq q\}$). Analogous notions appear
substituting the word ``future" by ``past" and, denoting
$I^-(p), J^-(p)$.

\begin{prop}\label{bolas}
Let $(S\times \R,  g)$ be a standard stationary spacetime as in
\eqref{e1} and $(x_0,t_0)\in S\times\R$. Then
 $$I^+(x_0,t_0)= \cup_{s> 0}\{t_0+s\}\times B_F^+(x_0,s), $$
 $$I^-(x_0,t_0)= \cup_{s< 0}\{t_0-s\}\times B_F^-(x_0,s).$$
\end{prop}
\begin{proof}
See \cite[Proposition 4.2]{CJS11}.
\end{proof}
Using the expression of the chronological future and past in terms of the forward and backward 
balls of the Fermat metric, we can easily obtain
the characterization of the causality conditions in terms of the Fermat metric. For definitions 
and properties
of the different levels of causality we refer to \cite{MinSan08}.
\begin{thm}\label{causalstructure}
Let $(S\times \R,  g)$ be a standard stationary spacetime as in \eqref{e1}. Then
$(S\times \R,  g)$ is causally continuous. Furthermore,
\begin{enumerate}[(a)]
 \item[(a)]
 it is causally simple if and only if one of the following equivalent conditions holds:
\begin{enumerate}
\item[(i)] $J^+(p)$ is closed
for all $p$, 
\item[(ii)] $J^-(p)$ is closed for all $p$ and 
\item[(iii)] the
associated Finsler manifold $(S,F)$ is convex,
\end{enumerate}
\item[(b)] it is globally hyperbolic  if and only
if the subsets $B^+_F(x,r)\cap B^-_F(x,r)$ are relatively compact for
every $x\in S$ and $r>0$.
\end{enumerate}
Moreover,  a
slice $S\times \{t_0\}, t_0\in\R$, is a Cauchy hypersurface if and
only
if the Fermat metric $F$ on $S$
is forward and backward complete.
\end{thm}
\begin{proof}
See \cite[Theorems 4.3 and 4.4]{CJS11}. For part $(b)$ see also \cite[Proposition 2.2]{CJS11}.
\end{proof}
The static version of the last proposition can be found in \cite[Proposition 3.5]{San97} (see also
\cite[Theorem 3.66]{BeEr}). Furthermore, an extension of last theorem characterizing
the stationary regions that are causally simple in terms of convex regions for the Fermat metric has been achieved in \cite{CGS11}.
Theorem \ref{causalstructure} implies some consequences for Randers metrics. In particular we can establish
a generalization of the classical Hopf-Rinow theorem.
\begin{thm}\label{hopf-rinow}
Given a Randers manifold $(M,R)$, the following conditions are equivalent:
\begin{enumerate}[(i)]
\item the subsets $B^+_R(x,r)\cap B^-_R(x,r)$ are relatively compact for
every $x\in M$ and $r>0$,
\item the subsets that are forward and backward bounded are relatively compact,
\item there exists $f:M\rightarrow \R$ such that $R+df$ is a forward and backward complete
Randers metric.
\end{enumerate}
Moreover, these conditions imply the convexity of $(M,R)$.
\end{thm}
\begin{proof}
The equivalence between the two first conditions is standard and it holds for any Finler metric. For $(i)\Rightarrow (iii)$, first observe that any Randers metric can be obtained as the Fermat metric of a standard stationary spacetime (see \cite[Proposition 3.1]{BiJa11}). Now let $(S\times \R,g)$ be the standard stationary spacetime having as a Fermat metric $R$. By Proposition \ref{causalstructure}, this spacetime is globally hyperbolic, but then using \cite{BerSan03}, we obtain that there exists a smooth spacelike Cauchy hypersurface $S_f$. Consider the splitting associated to the Cauchy hypersurface. By Proposition \ref{F+df}, the Fermat metric associated to the new splitting is of the form $R-df$ for a certain smooth function $f:S\rightarrow \R$. Moreover, by Proposition \ref{causalstructure}, $R-df$ must be forward and backward complete. For $(iii)\Rightarrow (i)$, observe that part $(b)$ of Theorem \ref{causalstructure} implies that the stationary spacetime associated to
$R+df$ is globally hyperbolic. By Proposition \ref{F+df}, the stationary spacetime associated to $R$ is the same as the one associated to $R+df$, but considering another splitting. Therefore it is globally hyperbolic and 
$(i)$ follows from part $(b)$ of Theorem \ref{causalstructure}. The convexity can be obtained from the Avez-Seifert Theorem applied to $(S\times \R,g)$ (see for example \cite[Theorem 6.1]{BeEr}).
\end{proof}
It turns out that the condition of forward or backward completeness can be substituted by one of the two first equivalent conditions in Theorem \ref{hopf-rinow} in some classical results of Finsler Geometry, as for example, the theorems of Bonnet-Myers and Synge or the sphere theorem in its 
non-reversible version by Rademacher (see \cite[Remark 5.3]{CJS11}). 

Theorem \ref{causalstructure} has been used in \cite{DPS11} to obtain some conditions that ensure
global hyperbolicity. Recall that $h$ is defined in \eqref{perfermatmetric}. Given any Riemannian metric
$g$ in $S$, we will denote by $d_g$ the distance in $S$ associated to $g$. We say that a positive function $f$ in $S$ grows at most linearly with respect to $d_g$ if given a point $x_0$, there exist positive constants $A,B$ such that $f(x)\leq A\cdot d_g(x_0,x)+B$ for every $x\in S$. This condition does not depend on $x_0$. We also will denote 
\[\|\omega\|_{g}=\sup_{v\in T_xS}\frac{|\omega(v)|}{\sqrt{g(v,v)}}\]
the $g$-norm of a one-form $\omega$ in $x\in S$ for any Riemannian metric $g$ on $S$.
\begin{thm}
Let $(S\times \R,g)$ be a conformally standard stationary spacetime with $g$ as in \eqref{e1conf}. Then the slices 
$S\times \{t\}$, $t\in\R$, are Cauchy hypersurfaces if one of the following conditions hold:
\begin{enumerate}[(i)]
\item the metric $\frac{1}{(1+\|\omega\|^2_{g_0})^2}h$ is complete,
\item the metric $g_0$ is complete and $\|\omega\|_{g_0}$ grows at most linearly in $d_{g_0}$,
\item there exists a proper function $f:S\rightarrow\R$ such that the
product $\|df\|_{g_0}\cdot\|\omega\|_{g_0}$ grows at most linearly in $d_{g_0+df\otimes df}$
\end{enumerate}
Moreover, if $(S\times \R,g)$ is globally hyperbolic,
\begin{enumerate}[(i)]
\item[(iv)] the slices $S\times\{t\}$, $t\in\R$, are Cauchy hypersurfaces if $\|\omega\|_{g_0}^2$ grows at most linearly
in $d_h$,
\item[(v)] for any proper function $f:S\rightarrow \R$, $\|\omega\|_{g_0}$ grows at most linearly
in $d_{g_0+df\otimes df}$.
\end{enumerate}
\end{thm}
\begin{proof}
For $(i)$ and $(iv)$ see \cite[Theorem 2]{DPS11}. For $(ii)$, see part (1) of Proposition 2 in \cite{DPS11} and for $(iii)$ and $(v)$, \cite[Theorem 4]{DPS11}.
\end{proof}
Indeed, in \cite{DPS11}, the authors obtain several interesting pinching inequalities as
\[\frac{\sqrt{h(v,v)}}{2(1+\|\omega\|_{g_0}^2)}\leq F(v)\leq 2\sqrt{h(v,v)}\]
and
\[\frac{\sqrt{g_0(v,v)}}{\sqrt{1+\|\omega\|^2_{g_0}}+\|\omega\|_{g_0}}\leq F(v)\leq (\sqrt{1+\|\omega\|^2_{g_0}}+\|\omega\|_{g_0})\sqrt{g_0(v,v)}\]
for every $v\in TS$ (see \cite[Propositions 1 and 2]{DPS11}).
We point out that in \cite{San97} (especially in Corollary 3.5) there are some results in the same direction as the last theorem.

As a further relation between Causality of a standard stationary spacetime and 
Randers metrics, Cauchy developments will
be constructed in terms of the Fermat metric. 
A subset $A$ of a
spacetime $M$ is {\it achronal} if no  $x,y\in A$ satisfies $x\ll
y$; in this case,  the {\it future (resp. past) Cauchy
development} of $A$, denoted by $D^+(A)$ (resp. $D^-(A)$), is the
subset of points $p\in M$  such that every  past- (resp. future)-
inextendible causal curve through $p$ meets $A$. The union
$D(A)=D^+(A)\cup D^-(A)$ is the {\it Cauchy development} of $A$.
The {\it future  (resp. past) Cauchy horizon $H^+(A)$ (resp. $H^-(A)$)} is
defined as
\[H^\pm(A)=\{p\in \bar{D}^{\pm}(A):I^\pm(p)\,\,\text{does not meet $D^\pm(A)$}\}.\]
Intuitively, $D(A)$ is the region of $M$ a priori predictable from
data in $A$,  and its  {\em horizon} $H(A)=H^+(A)\cup H^-(A)$, the
boundary of this region.

\begin{prop}\label{p4.5}
Let $(S\times \R,  g)$ be a standard stationary spacetime as in
\eqref{e1} such that  $S\times\{t_0\}$ is a Cauchy hypersurface,  $A\subset
S$, and $A_{t_0}=A\times \{t_0\}$  the corresponding (necessarily
achronal) subset of $S\times \{t_0\}$. Then
\begin{equation}\label{cauchydevelop}
D^+(A_{t_0})=\{(y,t):d_F (x,y)>t-t_0\,\,\text{for every $x\notin
A$ and $t\geq t_0$}\},
\end{equation}
\begin{equation}\label{cauchydevelop-}
D^-(A_{t_0})=\{(y,t):d_F(y,x)>t_0-t\,\,\text{for every $x\notin
A$ and $t\leq t_0$}\}.\end{equation} 
Moreover, the Cauchy horizons can be described as
\begin{equation}\label{horizon}
H^+(A_{t_0})=\{(y,t):\inf_{x\notin A}d_F(x,y)=t-t_0
\}
\end{equation}
\begin{equation}\label{horizon-}
\quad H^-(A_{t_0})=\{(y,t):\inf_{x\notin A} d_F(y,x)=t_0-t
\}.
\end{equation}
\end{prop}
\begin{proof}
See \cite[Proposition 4.7]{CJS11}.
\end{proof}
Last proposition can be used to study the differentiability of the Cauchy horizon in terms of the
distance computed with the Fermat metric from a closed subset \cite[Theorem 4.10]{CJS11} and vice versa (see
\cite[Subsection 5.4]{CJS11}.
\section{Causal boundaries and Fermat metrics}
In General Relativity it is important to complete the spacetime with some kind of boundary. 
One way to obtain an intrinsic completion is using the causal structure. This completion
has been largely discussed because of some topological pathologies, but it seems that after 
\cite{FHS10a} the definition and the topology is now satisfactory. As the causal completion (or c-completion for short) depends only on the causal structure of the spacetime, it is expectable
that in conformally standard stationary spacetimes can be computed in terms of Fermat metric. 
Let us recall that the c-completion is constructed in strongly causal spacetimes by adding some ideal points to the spacetime in such a way that timelike curves always have an endpoint in the new space. This is done by identifying the points of the space with PIP's (resp. PIF's), that is, proper indecomposable past (resp. future) sets, in other words, every point $p\in M$ is identified with $I^-(p)$ and $I^+(p)$. Then we add to the spacetime the
TIP's (resp. TIF's), that is, terminal indecomposable past (resp. future) sets. Then the future (resp. past) c-completion $\hat{\partial} M$ (resp. $\check{\partial} M$) is given  by the TIP's (resp. TIF's). Moreover, $\hat{M}:=M\cup\hat{\partial} M$ and $\check{M}:=M\cup\check{\partial} M$.  In order to obtain the causal completion of $M$, we must identify some TIP's and TIF's. This is done by means of the $S$-relation. Denote $\hat{M}_\emptyset=\hat{M}\cup \{\emptyset\}$ (resp. $\check{M}_\emptyset=\check{M}\cup \{\emptyset\}$).
The $S$-relation is defined in $\hat{M}_\emptyset\times \check{M}_\emptyset$ as follows. If $(P,F)\in \hat{M}\times\check{M}$, then $P\sim_S F$ if and only if
\begin{itemize}
\item[(i)]
$F$ is included and a maximal indecomposable future set in $\uparrow P$ (the common future of $P$) and
\item[(ii)] $P$ is included and a maximal indecomposable past set in $\downarrow F$ (the common past of $F$).
\end{itemize}
Moroever,  we also put
\[P\sim_S \emptyset, \quad\quad \emptyset\sim_S F.\]
In particular, the only $S$-relations between PIP's and PIF's are $I^+(p)\sim_S I^-(p)$. Then the c-completion $\bar{M}$ is the quotient set $\hat{M}_\emptyset\times\check{M}_\emptyset/\sim_S$ endowed with the {\em chronological topology} (see Definition 2.2 and the paragraph below in \cite{FHS10b}). We can indentify
$M\equiv \{(I^-(p),I^+(p)):p\in M\}$ and define the {\em c-boundary} as $\partial M:=\bar{M}\setminus M$. We say that the c-completion is simple as a point set when every TIP (resp. TIF) determines a unique pair in $\partial M$ (for topologically simple see \cite[Definition 2.4]{FHS10b}).

Up to the completion of the Finsler manifold $(S,F)$, there are several non-equivalent ways to do it. You can compute the forward (resp. backward) Cauchy boundary $\partial^+_C S$ (resp. $\partial^-_C S$) by adding ideal points in such a way that you can always obtain the convergence of forward (resp. backward) Cauchy sequences.
Then the forward (resp. backward) {\em Cauchy completion} is $S_C^+:=S\cup \partial_C^+S$ (resp.
$S_C^-:=S\cup \partial_C^-S$). 
Moreover, denote $\partial_C^sS:=\partial^+_C S\cap \partial^-_C S$ and $S_C^s=S\cup \partial_C^sS$. 
The map 
\[d_Q:S_C^+\times (S_C^+\cup S_C^-)\rightarrow [0,\infty]\]
 defined by 
\[d_Q([\{x_n\}],[\{y_m\}]):=\lim_n(\lim_m d_F(x_n,y_m))\]
is a quasi-distance (see \cite[Propositions 3.25 and 3.32]{FHS10b}). 

The {\em Gromov completion} is obtained by considering the subset ${\mathcal L}_1(S,d_F)$ of $d_F$-Lipschitz functions on $S$ \cite[Definition 5.2]{FHS10b} and ${\mathcal L}_1(S,d_F)_*={\mathcal L}_1(S,d_F)/\R$ (two functions are related when they differ in a constant). Then define the maps
\begin{align*}
j^+:S_C^+\rightarrow {\mathcal L}_1(S,d_F), \quad x\to -d_x^+,\quad \text{where $d_x^+=d_Q(\cdot,x)$},\\
j^-:S_C^-\rightarrow {\mathcal L}_1(S,d_F), \quad x\to +d_x^-,\quad \text{where $d_x^-=d_Q(x,\cdot)$},
\end{align*}
which are injective (see \cite[Proposition 5.7]{FHS10b}). We can
identify the points of $S$ with the class of  (minus) the distance function to (resp. from) the point, which is denoted by $j^+(S)$ (resp. $j^-(S)$). The forward (resp. backward) Gromov completion $S_G^+$ (resp. $S_G^-$) is the closure of $S$ in ${\mathcal L}_1(S,d_F)_*$ considering the compact-open topology. Observe that this topology is equivalent to that of convergence on compact subsets and to that of
pointwise convergence. 

Let $C^+(S)$ be the set of piecewise smooth curves $c:[\alpha,+\Omega)\rightarrow S$, $\Omega\leq \infty$, such that $F(\dot c)<1$. For $c\in C^+(S)$, the associated (forward) Busemann function $b_c^+:S\rightarrow (-\infty,\infty]$ is
$b_c^+(\cdot)=\lim_{s\to\Omega}(s-d_F(\cdot,c(s)))$, which always exists because is increasing
\cite[Lemma 4.14]{FHS10b}. Observe that $b_c^+$ is finite or infinite everywhere. Denote
\[B^+(S):=\{b_c^+<\infty : c\in C^+(S)\}.\]
Moreover, if $\Omega<\infty$, then there exists some $\bar{x}\in S_C^+$ such that $b_c^+(x)=\Omega-d_F(x,\bar{x})$ for all $x\in S$ (here $d_F$ extended to $S_C^+$), and we denote
\[b_c^+=d_p^+:=\Omega-d_F(\cdot,\bar{x}),\]
with $p=(\bar{x},\Omega)\in S_C^+\times\R$.
If $\Omega=\infty$, we say that $b_c^+$ is a properly Busemann function and we write
\[{\mathcal B}^+(S):=\{b_c^+<\infty: c\in C^+(S),\Omega=\infty\}.\]
The {\em Busemann completion} as a point set is the quotient $S_B^+:=B^+(S)/\R\subset S_G^+$ and the (forward) {\em Busemann boundary}
$\partial_B^+S:=S_B^+\setminus S (\subset \partial_G^+S)$. Furthermore, the (forward) {\em properly Busemann boundary} is defined as $\partial_{\mathcal B}^+S:={\mathcal B}^+(S)/\R$. Then $\partial_B^+S=j^+(\partial_C^+S)\cup\partial_{\mathcal B}^+S$. $S_B^+$ will be endowed with the chronological topology (see \cite[Subsection 5.2.2]{FHS10b}). 

Recall that given a topological space $T$, the {\em forward cone} with base $T$ is constructed as the quotient topological space $(T\times (-\infty,\infty])/\sim$, where the unique non-trivial identifications are $(x,\infty)\sim (x',\infty)$ for all $x,x'\in T$. Moreover, the class of $(x,\infty)$ is called the {\em apex} of the cone.

Finally, given a future-pointing timelike curve $\gamma:[\alpha,\Omega)\rightarrow M$, parametrized as $\gamma(t)=(c(t),t)$, observe that $I^-[\gamma]=\{(x',t')\in M : t'<b_c^+(x')\}$. Therefore the indecomposable past sets ($\not=M$) can be identified with $B^+(S)$. When $b_c^+\equiv\infty$, then $I^-[\gamma]=M$ and it will be denoted with $i^+$. If $(P,F)\in\partial M$ with $\emptyset\not=P=P(b_c^+)$, we define the line over $(P,F)$, denoted as $L(P,F)$ as 
\begin{itemize}
\item If $F=\emptyset$ then $L(P,\emptyset):=\{(P',\emptyset):P'=P(b_c^+ +k),k\in\R\}$,
\item If $F\not=\emptyset$, it follows that $P=P(d_p^+)$ and $F=F(d_{p'}^-)$ (see \cite[Theorem 6.15]{FHS10b}), and then $L(P,F):=\{(P',F'):P'=P(d_p^+ +k), F'=F(d_{p'}^-+k), k\in\R\}.$ 
\end{itemize}
A dual definition is assumed for $(P,F)\in\partial M$ with $\emptyset\not= F=F(b_c^-)$.

Let us first describe the structure of the c-completion as a point set.
\begin{thm}\label{th} Let $(S\times\R,g)$ be a (conformally) standard stationary spacetime
as in \eqref{e1conf} and denote $M=S\times\R$.  Then,
the c-boundary $\partial M$ has the following structure:
 \begin{itemize} 
 \item[(i)] The future
(resp. past) c-boundary $\hat{\partial} M$ (resp. $\check{\partial} M$) is naturally a point set cone with base $\partial_B^+ S$ (resp. $\partial_B^-
S$) and apex $i^+$ (resp. $i^-$). 
\item[(ii)] A pair $(P,F)\in
\partial V$ with $P\neq \emptyset$ satisfies that $P=P(b_c^+)$ for
some $c\in C^+(S)$ and:
\begin{itemize}
\item[(a)] If $b_c^+\equiv \infty$ then $P=M, F=\emptyset$.
\item[(b)] If $b_{c}^+\in {\mathcal B}^+(S)(\equiv
\partial_{\mathcal B}^+S\times\R)$ then $F=\emptyset$. 
\item[(c)] If
$b_c^+\in B^+(S)\setminus {\mathcal B}^+(S)$, then $b_c^+=d_p^+$ with $p=(x^+,\Omega^+)\in 
\partial_C^+S\times\R$, $P=P(d^+_p)$ and $F\subset
F(d_{p}^-)$. In this case, there are two exclusive possibilities:
\begin{itemize}
\item[(c1)] either $F=\emptyset$, 
\item[(c2)] or $F=F(d_{p'}^-)$
with $p'=(x^-,\Omega^-)\in \partial_C^-S\times\R$ and satisfying
\[\Omega^--\Omega^+=d_Q(x^+,x^-)\] (in this case, $p'$ is not necessarily
unique).
\end{itemize}
Moreover, if $x^+\in \partial_C^s S$, then $p'=p$, $\uparrow
P=F(d_{p}^{-})$ and $P$ is univocally S-related with $F=F(d_p^+)$.
\end{itemize}
A dual result holds for pairs $(P,F)$ with $F\neq \emptyset$. So,
the total c-boundary is the disjoint union of lines $L(P,F)$.

When $\partial M$ is simple as a point set, it is the quotient set $\hat{\partial} M\cup_d \check{\partial} M/\sim_S $ of the partial boundaries $\hat{\partial} M, \check{\partial} M$  under the S-relation.

\end{itemize}
\end{thm}
\begin{proof}
See \cite[Theorem 1.2]{FHS10b}.
\end{proof}
Let us finally describe the causal and topological structures. As to the completions of $(S,F)$, let 
us remark that in the description of the c-completion, we only need Busemann and Cauchy completions,
while Gromov completion is useful to define the Busemann one. Observe that Gromov completion is a compact metrizable topological space and the Busemann one is $T_1$, sequentially compact but not necessarily Hausdorff. This is because,
in $S_B^+$, the topology inherited from the Gromov completion is finer than the chronological topology (otherwise the Busemann completion would not be sequentially compact). As a matter of fact, Busemann and Gromov completions coincide both as a point set and as a topological space when $S_B^+$ is Hausdorff.
\begin{thm}Let $(S\times\R,g)$ be a (conformally) standard stationary spacetime
as in \eqref{e1conf} and denote $M=S\times\R$.  Then, for each  $(P,F)\in \partial M$,
the line $L(P,F)$ is 
\begin{itemize}
\item[(i)] timelike if $P=P(d_p^+)$ and $F=F(d_p^-)$ for
some $p\in  \partial_C^s M\times \R$, 
\item[(ii)] horismotic if either $P$ or
$F$ are empty and 
\item[(iii)] locally horismotic otherwise 
\end{itemize}
(see \cite[Definition 6.22]{FHS10b}).

As to the topology of the c-completion:
\begin{itemize}\item[(iv)] If $S_B^+ $ (resp.
$S_B^- $) is Hausdorff, the future (resp. past) causal boundary
has the structure of a (topological) cone  with base $\partial_B^+S$ (resp. $\partial_B^-S$)
and apex $i^+$ (resp. $i^-$). 
\item[(v)] If $S_{C}^{s}$ is
locally compact and $d_Q^+$ is a generalized distance, then
$\overline{M}$ is simple  and so, it coincides with the quotient
topological space $\hat{M}\cup_d \check{M}/\sim_S$ of the partial
completions $\hat{M}$ and $\check{M}$ under the S-relation.
\end{itemize}
Summarizing, if $S_C^s $ is locally compact, $d_Q$ is a
generalized distance and $S_B^\pm $ is Hausdorff, $\partial M$
coincides with the quotient topological space
$(\hat{\partial}M\cup_d \check{\partial}M)/\sim_{S}$, where
$\hat{\partial}M$ and $\check{\partial}M$ have the structure of
cones with bases $\partial_B^+ S, \partial_B^- S$ and apexes
$i^+,i^-$, resp.
\end{thm}
\begin{proof}
See \cite[Theorem 1.2]{FHS10b}.
\end{proof}
\section{Existence of lightlike geodesics}
The study of multiplicity of lightlike geodesics between an event and a vertical line was the original
scope of the use of Fermat metrics. For example, in \cite{MaPi98}, the authors use the shortening
method applied to the Fermat metric to give some existence results. It is remarkable that in \cite{MaPi98} the authors refer to the Fermat metric as a pseudo-Finsler metric and they are concerned about the local existence, uniqueness and regularity of minimizers of the length functional
(see \cite[Appendix A.1]{MaPi98}). Of course, this is because, in that moment, they were not aware of the fact that Randers metrics are fiberwise strongly convex. By the way, it seems that this fact is not collected in the classical books of Finsler geometry available at the time that \cite{MaPi98} was published. 
This was done just two years later in \cite[Section 11.1]{BaChSh00}. 

Once you know that Randers metrics are fiberwise strongly convex, the local existence and uniqueness of geodesics are guaranteed. Moreover, studying lightlike geodesics between an event and a vertical line or ligthlike
geodesics spatially closed in a conformally standard stationary spacetime is equivalent to studying the existence and multiplicity of geodesics between two points or closed geodesics of the Fermat metric, respectively (up to the case with boundary see \cite{CGS11} or \cite{CapPro} in these proceedings). This can be done by applying the theories of Lyusternik-Schnirelmann and Morse to the energy functional of a Finsler manifold $(M,F)$.  In fact, you can consider the space of curves of Sobolev class $H^1$ on $M$. 
Recall that this space does not depend on the Riemannian metric that you fix on $M$. Thus we fix an auxiliary Riemannian metric $h$ on $M$. Moreover, fix a smooth submanifold $N$ of $M\times M$ and consider the collection  $\Lambda_N(M)$ of the curves $x:[0,1]\rightarrow M$, having $H^1$ regularity, that is, $x$ is absolutely continuous and $\int_0^1 h(\dot x,\dot x)\df s$ is finite, and 
with $(x(0),x(1))\in N\subseteq M\times M$. Then it is well-known that $\Lambda_N(M)$ is a Hilbert manifold modeled on any of the equivalent Hilbert spaces of all the $H^1$-sections with endpoints in $TN$ of the pulled back bundle $x^*TM$, with $x$ a regular curve in $\Lambda_N(M)$. Let us observe 
that even when the strong convexity condition is available, we must pay some attention to the fact that
$F^2$ is not even $C^2$ on the zero section unless $F^2$ is quadratic, i.e. a Riemannian metric (see \cite{Warner65}).

\begin{prop}\label{regularity}
A   curve  $\gamma\in\Lambda_N(M)$ is  a geodesic for the Finsler manifold $(M,F)$ satisfying 
\begin{equation}\label{orto}
g_{\dot\gamma(0)}(V,\dot \gamma(0))=g_{\dot\gamma(1)}(W,\dot \gamma(1))
\end{equation}
for any $(V,W)\in T_{(\gamma(0),\gamma(1))}N$
if and only if it is a (non constant) critical point of the energy functional $E_F$ on $\Lambda_N(M)$.
\end{prop}
\begin{proof}
See for example \cite[Proposition 2.1]{CJM11}.
\end{proof}
Moreover, recall that a funcional $J$ defined on a Banach manifold $(X,\|\cdot\|)$ satisfies the Palais-Smale condition if every sequence $\{x_n\}_{n\in\N}$ such that $\{J(x_n)\}_{n\in\N}$ is bounded and $\|dJ(x_n)\|\rightarrow 0$ contains a convergent subsequence. This condition is fundamental
to apply the theories of Lyusternik-Schnirelmann and Morse, which study the relation between the number of critical points and the topology of the manifold. Palais-Smale is satisfied by the energy functional precisely when one of the equivalence conditions of the Generalized Hopf-Rinow Theorem in \ref{hopf-rinow} holds.
 \begin{thm}\label{ps} Let $(M,F)$ be a Finsler manifold with $B_F^+(x,r)\cap B_F^-(x,r)$ relatively compact for every $x\in M$ and $r>0$,
 and $N$,  a closed submanifold on
$M\times M$ such that the first or the second projection of $N$ to $M$ is compact, 
then $E_F$ satisfies the
Palais-Smale condition on $\Lambda_N(M)$. 
\end{thm}
\begin{proof}
See \cite[Theorem 3.1]{CJM11} or \cite{Mer78} 
and the comments before \cite[Theorem 5.2]{CJS11}. 
\end{proof}
Again the most difficult part to prove Palais-Smale for the energy functional of a Finsler metric is the lack of differentiability of $F^2$ in the zero section. As $F^2$ is not $C^2$ on the zero section, we can only apply the mean value theorem to the derivatives of $F^2$ away from the zero section. Once Palais-Smale condition is available, we can apply Lyusternik-Schnirelmann theory to obtain multiplicity results of geodesics between two arbitrary points when the manifold is non-contractible (see \cite[Proposition 3.1]{CJM11}). With a different approach, it is possible to prove the existence of only a finite number of geodesics between two non-conjugate points in the presence of a convex function for the Finsler metric \cite[Theorem 2.4]{CJM10}.
\subsection{Morse Theory for lightlike geodesics}
As to the Morse theory for the energy functional in the space of $H^1$-curves, the main difficulty is that $E_F$ is not $C^2$ even in a geodesic unless the restriction of $F^2$ to the geodesic is quadratic (see \cite{AbbSch09}
and also \cite{Cap10}). As a consequence, Morse Lemma cannot be proved in the curves of class $H^1$ with the standard techniques. Even if Morse Theory works for $C^{1,1}$-functionals in Hilbert manifolds (see for example \cite[Chapter 8]{MawWil89}), the  Morse Lemma is
essential to compute the critical groups in terms of the index of the critical point. 
In \cite{CJM10b}, this problem is circumvented using that the space of curves 
with $C^1$ regularity is a Banach manifold densely immersed in the Hilbert space of $H^1$ curves and
$E_F$, restricted to the $C^1$ class, admits second differential in regular curves of $C^1$. To be more precise, consider the second differential of $E_F$ in the space of $C^1$-curves, assume for simplicity that the kernel is trivial and extend it by density to $H^1$. This gives a functional that it is
reprensented by the identity plus a compact operator in a certain scalar product \cite[Lemma 2]{CJM10b}. Moreover, the restriction of this operator to the space of $C^1$-curves gives an invertible operator
\cite[Lemma 5]{CJM10b}, then one can obtain a Morse Lemma for this restriction and the scalar product of $H^1$ \cite[Theorem 7]{CJM10b}. Finally we show that the critical groups of the $C^1$-class coincide with those of $H^1$ using a classical result by Palais. As the geometrical index of a lightlike geodesic coincides with the one of its projection as a Fermat geodesic \cite[Theorem 13]{CJM10b}, 
the Morse relations for lightlike geodesics in conformally standard stationary spacetimes follow.
\begin{thm}\label{Morseluce}
Let $(S\times\R,g)$ be a globally hyperbolic conformally standard stationary spacetime with $g$ as in \eqref{e1conf}, $p=(p_0,t_0)\in S\times\R$ and $L_{q_0}=\{(q_0,s)\in S\times\R:s\in\R\}$. Assume that for each $s\in\R$ the points $p$ and $(q_0,s)$ are non-conjugate along every future-pointing lightlike geodesic connecting them. Then there exists a formal series $Q(r)$ with coefficients in $\N\cup\{+\infty\}$ such that
\begin{equation*}\label{relazioni}
\sum_{z\in G_{p,L_{q_0}}}r^{\mu(z)}=P(r,\Lambda_{(p_0,q_0)}(S))+(1+r)Q(r),
\end{equation*}
where  $G_{p,L_{q_0}}$ is the set of all the future-pointing lightlike geodesics connecting $p$ to $L_{q_0}$, $\mu(z)$ is the number of conjugate points of $z$ counted with multiplicity and $P(r,\Lambda_{(p_0,q_0)}(S))$ is the Poincar\'e polynomial of $\Lambda_{(p_0,q_0)}(S)$.
\end{thm}
\begin{proof}
See \cite[Theorem 15]{CJM10b}. 
\end{proof}
A Gromoll-Meyer type theorem to obtain a multiplicity results for geodesics joining two points $p$ and $q$ that excludes the case in that
all the geodesics come from the iterations of a closed geodesic that goes through $p$ and $q$
(as in the round sphere) is obtained in \cite{CaJa11} in the case that $p$ and $q$ are non-conjugate.

Let us observe that even if the problem of existence of normal geodesics between two arbitrary
submanifolds in a standard stationary spacetime cannot be reduced to a problem for the Fermat metric
in general, in \cite{BCC10}, the authors use
completeness of the Fermat metric to prove a result of this type with some hypotheses in the submanifolds \cite[Theorem 1.1]{BCC10}.

\subsection{$t$-periodic lightlike geodesics and the closed geodesic problem}
Let us recall that a lightlike geodesic $\gamma=(x,t):\R\rightarrow S\times\R$  in a standard stationary spacetime is said $t$-periodic if there exists $T\geq 0$ and $s_0>0$ such that $x$ is periodic, that is, $x$ and its derivatives coincide in $0$ and $s_0$,
$t(s_0)=t(0)+T$ and $\dot t(s_0)=\dot t(0)$. In this case, $T$ is called the {\em universal period}.
 They are related with closed geodesics for the Fermat metric.
\begin{prop} Let $(M,g)$ be a conformally standard stationary spacetime as in \eqref{e1conf}. Then $\gamma=(x,t):\R\rightarrow S\times\R$ is a $t$-periodic lightlike geodesic if and only if 
$x:\R\rightarrow S$ is a closed geodesic of the Fermat metric. 
\end{prop}
\begin{proof}
The implication to the right follows from Proposition \ref{ppp}. For the other one, first observe that
$g(\dot\gamma,\partial_t)$ is constant. To see this, recall that as $\partial_t$ is a conformal field, it satisfies
\begin{equation}\label{conformaleq}
g(\nabla_V\partial_t,W)+g(\nabla_W\partial_t,V)=\lambda g(V,W)
\end{equation}
for every $V,W\in \mathfrak{X}(M)$, where $\nabla$ is the Levi-Civita connection of $(M,g)$
and $\lambda$ a positive function on $M$.
Then, using that $\gamma$ is a lightlike geodesic and \eqref{conformaleq},
\[\frac{d}{ds}g(\dot\gamma,\partial_t)=g(\dot\gamma,\nabla_{\dot\gamma}\partial_t)=\frac 12 \lambda g(\dot\gamma,\dot\gamma)=0.\]
Using again Proposition \ref{ppp}, we deduce that there exists $s_0>0$ such that $\gamma(s_0)=\gamma(0)$ and $\dot\gamma(s_0)=\mu\dot\gamma(0)$ for some $\mu>0$, but the fact that  $g(\dot\gamma,\partial_t)$ is
a non null constant (it cannot be zero because $\dot\gamma$ is lightlike and $\partial_t$ timelike) implies that $\mu=1$ as required.
\end{proof}

As to the closed geodesic problem, most of the classical results for Riemannian metrics, such as
Gromoll-Meyer and Bangert-Hingston theorems, are available in the Finslerian setting, obtaining the corresponding results of multiplicity for $t$-periodic lightlike geodesics in 
conformally standard  stationary spacetimes (see \cite{BiJa11} and references therein). As an exception, there are Finsler metrics with a finite number of geometrically distinct closed geodesics, the so-called Katok metrics. Remarkably these Finsler metrics are of Randers type and they have constant flag curvature. Let us observe that the classification of Randers metrics of
constant flag curvature has been obtained in \cite{BRS04} using the expression of a Randers metric as a Zermelo one, that is, a metric defined from a Riemannian metric $g$ and a vector field $W$ in a manifold $M$ as
\[Z(v)=\sqrt{\frac{1}{\lambda}g(v,v)+\frac{1}{\lambda^2}g(v,W)^2}-\frac{1}{\lambda}g(v,W),\]
where $\lambda=1-g(W,W)$ must be positive. Indeed, $(M,Z)$ has constant flag curvature if and only if $W$ is a homothety and $g$ has constant curvature \cite{BRS04}.
 We can then
construct standard stationary spacetimes with compact orbit manifold $S$ and a finite number of geometrically distinct $t$-periodic light rays (see  \cite[Propositions 3.1 and 3.4]{BiJa11}). 

\subsection{Alternative functional to energy}
Existence and multiplicity of Fermat geodesics can be studied by means of other functionals
rather than the energy one. In \cite{FGM95}, the authors use the functional defined as 
\begin{equation*}
J(x)=\sqrt{\int_0^1 h(\dot x,\dot x)\df s}+\int_0^1 \omega(\dot x)\df s
\end{equation*}
for every curve $x:[0,1]\rightarrow S$ of class $H^1$ with $h$ as in \eqref{perfermatmetric}. The advantage of this functional is that
it is $C^2$ on geodesics. Its critical points are Fermat geodesics parametrized with $h$-constant 
speed (see also \cite{Mas09}). This functional has also been used in \cite{GGP09,GiJa10} to obtain
a result of genericity of stationary spacetimes without conjugate lightlike geodesics between a fix event $p$ and a fix vertical line.
\section{Further applications}
\subsection{Randers metrics of constant flag curvature and stationary spacetimes}
Constant flag curvature plays the same role in Finsler geometry as sectional curvature in 
the Riemannian setting, that is, it is an important invariant related to the behaviour of geodesics.
Let us recall that Randers metrics with constant flag curvature have been classified in \cite{BRS04}. These metrics have already appeared in the context of Fermat metrics to provide examples
of spacetimes with a finite number of geometrically distinct $t$-periodic lightlike geodesics \cite[Propositions 3.1 and 3.4]{BiJa11}. Subsequently, these spacetimes were studied in \cite{GHWW09}. 
\begin{prop}
Let $(S\times \R,g)$ be a conformally standard stationary spacetime as in \eqref{e1conf} 
whose Fermat metric
is of constant flag curvature. Then $(S\times \R,g)$ is locally conformally flat.
\end{prop}
\begin{proof}
See \cite[Section II.E.2]{GHWW09}.
\end{proof}
The converse of the last proposition is not true in general, because for example, you can find a Randers metric of the form $\sqrt{g_0}+df$ with $g_0$ the Euclidean metric in $\R^n$ and $f:\R^n\rightarrow \R$ a smooth function, such that $\sqrt{g_0}+df$ does not have constant flag curvature (see \cite[Section 3.9 B]{BaChSh00}), and $\sqrt{g_0}+df$ is the Fermat metric associated to a certain splitting of Minkowski spacetime (see Proposition \ref{F+df}). Anyway it is expectable, as commented in \cite[Section II.E.2]{GHWW09}, 
that given a conformally flat stationary spacetime, you can find a spacelike section having as a Fermat metric a Randers metric with constant flag curvature.
Let us point out that in \cite{GHWW09}, the authors give several examples of Randers metrics coming from well-known stationary spacetimes. They also interpret Randers metrics as magnetic flows.
\subsection{Timelike geodesics with fixed proper time}
Let us also remark that existence of timelike geodesics with fixed proper time between an event and a vertical line in a standard stationary spacetime can be reduced to existence of lightlike geodesics in a one-dimensional higher standard stationary spacetime. Observe that, in this case, as timelike geodesics are not preserved by conformal changes, we cannot consider conformally standard stationary spacetimes as in \eqref{e1conf}, but standard stationary spacetimes $(S\times\R,g)$ such that
\[g((v,\tau),(v,\tau))= g_0(v,v)+2\omega(v)\tau-\beta \tau^2,\]
in $(x,t)\in S\times \R$ for any $(v,\tau)\in T_xS\times \R$, where $\omega$ and $g_0$ are respectively
a one-form and a Riemannian metric on $S$ and $\beta$ is a positive function on $S$. Then we can 
define the one-dimensional higher spacetime $(S\times \R^2,\eta)$, with $\eta$ defined as
\[\eta((v,y,\tau),(v,y,\tau))= g_0(v,v)+y^2+2\omega(v)\tau-\beta \tau^2,\]
in $(x,\nu,t)\in S\times\R^2$, where $(v,y,\tau)\in  T_xS\times\R^2$. A curve from the event $(x_0,t_0)\in S\times \R$ to the line $L_{x_1}=\{(x_1,s)\in S \times \R:s\in\R\}$
is a  timelike geodesic $\gamma=(x,t):[0,1]\rightarrow S\times \R$ of $(S\times \R,g)$
 with proper time $T$ if and only if $[0,1]\ni s\to (x(s),s,t(s))\in S\times \R^2$ is a lightlike
geodesic of $(S\times \R^2,\eta)$ from the event $(x_0,0,t_0)$ to the line $\{(x_1,T,s)\in S\times\R^2:s\in\R\}$. Moreover, the Fermat metric of this standar stationary spacetime is
given as
\[\tilde{F}(v,y)=\sqrt{\frac{1}{\beta}g_0(v,v)+\frac{v^2}{\beta}+\frac{1}{\beta^2}\omega(v)^2}+\frac{1}{\beta}\omega(v),\]
in $(x,\nu)\in S\times \R$, where $(v,y)\in T_xS\times \R$. As completeness conditions for the original Fermat metric in \eqref{fermatmetric}, which in this case is expressed as
\begin{equation*}\label{fermatmetricbeta}
F(v)=\sqrt{\frac{1}{\beta}g_0(v,v)+\frac{1}{\beta^2}\omega(v)^2}+\frac{1}{\beta}\omega(v),
\end{equation*}
 imply completeness conditions for $\tilde{F}$ (see the proof of \cite[Proposition 4.2]{CJM10b}), some multiplicity results \cite[Proposition 4.2]{CJM11} and Morse relations \cite[Theorem 18]{CJM10b} are available when the spacetime is globally hyperbolic. 
\subsection{Conformal maps and almost isometries}
Another interesting relation between Fermat metrics and conformally standard stationary spacetimes
occurs at the level of transformations (see \cite{JaLiPi11}). As Fermat metrics remain invariant by conformal changes in the conformally stationary spacetime, we need to consider conformal maps in the spacetime. Moreover, as we want to project these maps  into maps of the orbit manifold $S$, they have to preserve the conformal vector field $K$. Summing up, they have to be $K$-conformal maps, denoted by ${\rm Conf}_K(S\times\R,g)$, that is, they must preserve the metric up to a positive constant in every point and the conformal vector field $K$. As to General Relativity, these maps are precisely those that preserve the causal structure and the observers along $K$. Their counterpart in Fermat metrics are the so-called {\em almost isometries}, which are maps $\varphi:S\rightarrow S$ such that $\varphi^*(F)=F+df$ for a certain smooth function $f:S\rightarrow\R$ (here
$*$ denotes the pullback operator). Let us denote by $\widetilde{\rm Iso}(S,F)$ the group of almost isometries of $(S,F)$, which is a Lie group \cite{JaLiPi11}. 
\begin{thm}
Let $\psi:S\times \R\rightarrow S\times \R$ be a $K$-conformal map of a conformally standard stationary spacetime as in \eqref{e1conf}. Then there exist functions $\varphi:S\rightarrow S$ and
$f:S\rightarrow \R$ such that $\psi(x,t)=\big(\varphi(x),t+f(x)\big)$ and $\varphi$ is an almost isometry for the Fermat metric of $(S\times \R,g)$. Moreover, $\varphi$ is a Riemannian isometry for the metric $h$ in \eqref{perfermatmetric}
and the map $\pi: {\rm Conf}_K(S\times\R,g)\rightarrow \widetilde{\rm Iso}(S,F)$, defined as
$\pi(\psi)=\varphi$, is a Lie group homomorphism. The map can be projected to the quotient
\[\bar{\pi}: {\rm Conf}_K(S\times\R,g)/{\mathcal K}\rightarrow \widetilde{\rm Iso}(S,F),\]
(where ${\mathcal K}$ is the subgroup generated by the flow of $K$) and gives an isomorphism of Lie groups.
\end{thm}
\begin{proof}
See \cite{JaLiPi11}.
\end{proof}
 As a consequence of this relation, it follows a result of genericity of stationary spacetimes with
 discrete $K$-conformal group.
 \begin{cor}
Given a manifold $S$, for a \emph{generic} 
set of data $(g_0,\omega)$, the stationary metric $g=g(g_0,\omega)$
given in \eqref{e1}
on $S\times \R$ has discrete $K$-conformal group $\mathrm{Conf}_K(S\times\R,g)$.
\end{cor} 
\begin{proof}
See \cite{JaLiPi11}.
\end{proof}


\begin{thebibliography}{10}

\bibitem{AbbSch09}
{\sc A.~Abbondandolo and M.~Schwarz}, {\em A smooth pseudo-gradient for the
  {L}agrangian action functional}, Adv. Nonlinear Stud., 9 (2009),
  pp.~597--623.

\bibitem{ACL88}
{\sc M.~Abramowicz, B.~Carter, and J.-P. Lasota}, {\em Optical reference
  geometry for stationary and static dynamics}, Gen. Relativ. Gravit., 20
  (1988), pp.~1172--1183.

\bibitem{BaChSh00}
{\sc D.~Bao, S.-S. Chern, and Z.~Shen}, {\em An Introduction to
  {R}iemann-{F}insler geometry}, Graduate Texts in Mathematics,
  Springer-Verlag, New York, 2000.

\bibitem{BRS04}
{\sc D.~Bao, C.~Robles, and Z.~Shen}, {\em Zermelo navigation on {R}iemannian
  manifolds}, J. Differential Geom., 66 (2004), pp.~377--435.

\bibitem{BCC10}
{\sc R.~Bartolo, A.~M. Candela, and E.~Caponio}, {\em Normal geodesics
  connecting two non-necessarily spacelike submanifolds in a stationary
  spacetime}, Adv. Nonlinear Stud., 10 (2010), pp.~851--866.

\bibitem{BeEr}
{\sc J.~K. Beem, P.~E. Ehrlich, and K.~L. Easley}, {\em Global {L}orentzian
  geometry}, vol.~202 of Monographs and Textbooks in Pure and Applied
  Mathematics, Marcel Dekker Inc., New York, second~ed., 1996.

\bibitem{BerSan03}
{\sc A.~N. Bernal and M.~S{\'a}nchez}, {\em On smooth {C}auchy hypersurfaces
  and {G}eroch's splitting theorem}, Comm. Math. Phys., 243 (2003),
  pp.~461--470.

\bibitem{BiJa11}
{\sc L.~Biliotti and M.~{\'A}. Javaloyes}, {\em {$t$}-periodic light rays in
  conformally stationary spacetimes via {F}insler geometry}, Houston J. Math.,
  37 (2011), pp.~127--146.

\bibitem{Cap10}
{\sc E.~Caponio}, {\em The index of a geodesic in a {R}anders space and some
  remarks about the lack of regularity of the energy functional of a {F}insler
  metric}, Acta Math. Acad. Paedagog. Nyh\'azi. (N.S.), 26 (2010),
  pp.~265--274.

\bibitem{CapPro}
\leavevmode\vrule height 2pt depth -1.6pt width 23pt, {\em Infinitesimal of
  local convexity of a hypersurface in a semi-{R}iemannian manifold},
  arXiv:1201.0147v1 [math.DG],  (2011).

\bibitem{CGS11}
{\sc E.~Caponio, A.~Germinario, and M.~S\'anchez}, {\em Geodesics on convex
  regions of stationary spacetimes and {F}inslerian {R}anders spaces},
  arXiv:1112.3892v1 [math.DG],  (2011).

\bibitem{CaJa11}
{\sc E.~Caponio and M.~A. Javaloyes}, {\em A remark on the {M}orse {T}heorem
  about infinitely many geodesics between two points}, arXiv:1105.3923v1
  [math.DG],  (2011).

\bibitem{CJM10}
{\sc E.~Caponio, M.~A. Javaloyes, and A.~Masiello}, {\em Finsler geodesics in
  the presence of a convex function and their applications}, J. Phys. A, 43
  (2010), pp.~135207, 15.

\bibitem{CJM10b}
{\sc E.~Caponio, M.~{\'A}. Javaloyes, and A.~Masiello}, {\em Morse theory of
  causal geodesics in a stationary spacetime via {M}orse theory of geodesics of
  a {F}insler metric}, Ann. Inst. H. Poincar\'e Anal. Non Lin\'eaire, 27
  (2010), pp.~857--876.

\bibitem{CJM11}
\leavevmode\vrule height 2pt depth -1.6pt width 23pt, {\em On the energy
  functional on {F}insler manifolds and applications to stationary spacetimes},
  Math. Ann., 351 (2011), pp.~365--392.

\bibitem{CJS11}
{\sc E.~Caponio, M.~A. Javaloyes, and M.~S\'anchez}, {\em On the interplay
  between {L}orentzian {C}ausality and {F}insler metrics of {R}anders type},
  Rev. Mat. Iberoamericana, 27 (2011), pp.~919--952.

\bibitem{DPS11}
{\sc A.~Dirmeier, M.~Plaue, and M.~Scherfner}, {\em Growth conditions,
  riemannian completeness and lorentzian causality}, J. Geom. Phys.,  (2011).

\bibitem{DoKe78}
{\sc J.~Dowker and G.~Kennedy}, {\em Finite temperature and boundary effects in
  static space-times}, J. Phys. A: Math. Gen., 11 (1978), pp.~895--920.

\bibitem{FHS10b}
{\sc J.~L. Flores, J.~Herrera, and M.~S\'anchez}, {\em {G}romov, {C}auchy and
  causal boundaries for {R}iemannian, {F}inslerian and {L}orentzian manifolds},
  arXiv:1011.1154v2 [math.DG],  (2010).

\bibitem{FHS10a}
\leavevmode\vrule height 2pt depth -1.6pt width 23pt, {\em On the final
  definition of the causal boundary and its relation with the conformal
  boundary}, arXiv:1001.3270v3 [math-ph],  (2010).

\bibitem{FGM95}
{\sc D.~Fortunato, F.~Giannoni, and A.~Masiello}, {\em A {F}ermat principle for
  stationary space-times and applications to light rays}, J. Geom. Phys., 15
  (1995), pp.~159--188.

\bibitem{GGP09}
{\sc R.~Giamb{\`o}, F.~Giannoni, and P.~Piccione}, {\em Genericity of
  nondegeneracy for light rays in stationary spacetimes}, Comm. Math. Phys.,
  287 (2009), pp.~903--923.

\bibitem{GiJa10}
{\sc R.~Giamb{\`o} and M.~A. Javaloyes}, {\em Addendum to: {G}enericity of
  nondegeneracy for light rays in stationary spacetimes [mr2486667]}, Comm.
  Math. Phys., 295 (2010), pp.~289--291.

\bibitem{GiPe78}
{\sc G.~Gibbons and M.~Perry}, {\em Black holes and thermal green functions},
  Proc. R. Soc. London, Ser. A, 358 (1978), pp.~467--494.

\bibitem{GHWW09}
{\sc G.~W. Gibbons, C.~A.~R. Herdeiro, C.~M. Warnick, and M.~C. Werner}, {\em
  Stationary metrics and optical {Z}ermelo-{R}anders-{F}insler geometry}, Phys.
  Rev. D, 79 (2009), pp.~044022, 21.

\bibitem{Ing57}
{\sc R.~S. Ingarden}, {\em On the geometrically absolute optical representation
  in the electron microscope}, Trav. Soc. Sci. Lett. Wroc\l aw. Ser. B.,
  (1957), p.~60.

\bibitem{JaLiPi11}
{\sc M.~A. Javaloyes, L.~Lichtenfelz, and P.~Piccione}, {\em Almost isometries
  of non-reversible metrics with applications to stationary spacetimes},
  preprint,  (2011).

\bibitem{JaSan08}
{\sc M.~A. Javaloyes and M.~S{\'a}nchez}, {\em A note on the existence of
  standard splittings for conformally stationary spacetimes}, Classical Quantum
  Gravity, 25 (2008), pp.~168001, 7.

\bibitem{JaSan11}
\leavevmode\vrule height 2pt depth -1.6pt width 23pt, {\em On the definition
  and examples of {F}insler metrics}, arXiv:1111.5066v2 [math.DG],  (2011).

\bibitem{KN}
{\sc S.~Kobayashi and K.~Nomizu}, {\em Foundations of differential geometry.
  {V}ol. {I}}, Wiley Classics Library, John Wiley \& Sons Inc., New York, 1996.
\newblock Reprint of the 1963 original, A Wiley-Interscience Publication.

\bibitem{Kov90}
{\sc I.~Kovner}, {\em Fermat principle in gravitational fields}, Astrophys. J.,
  351 (1990), pp.~114--120.

\bibitem{LauLif62}
{\sc L.~Landau and E.~Lifshitz}, {\em The classical theory of fields}, Pergamon
  Press; Addison-Wesley, Oxford; Reading, MA, 2nd~ed., 1962.

\bibitem{LeCi18}
{\sc T.~Levi-Civita}, {\em La teoria di {E}instein e il principio di {F}ermat},
  Nuovo Cimento, 16 (1918), pp.~105--114.

\bibitem{LeCi27}
\leavevmode\vrule height 2pt depth -1.6pt width 23pt, {\em The Absolute
  Differential Calculus}, Blackie \& Son Limited, London and Glasgow, 1927.

\bibitem{Li55}
{\sc A.~Lichnerowicz}, {\em Th\'eories relativistes de la gravitation et de
  l'\'electromagn\'etisme. {R}elativit\'e g\'en\'erale et th\'eories
  unitaires}, Masson et Cie, Paris, 1955.

\bibitem{LiTh47}
{\sc A.~Lichnerowicz and Y.~Thiry}, {\em Probl\`emes de calcul des variations
  li\'es \`a la dynamique classique et \`a la th\'eorie unitaire du champ}, C.
  R. Acad. Sci. Paris, 224 (1947), pp.~529--531.

\bibitem{Mas09}
{\sc A.~Masiello}, {\em An alternative variational principle for geodesics of a
  {R}anders metric}, Adv. Nonlinear Stud., 9 (2009), pp.~783--801.

\bibitem{MaPi98}
{\sc A.~Masiello and P.~Piccione}, {\em Shortening null geodesics in
  {L}orentzian manifolds. {A}pplications to closed light rays}, Differential
  Geom. Appl., 8 (1998), pp.~47--70.

\bibitem{Mat72}
{\sc M.~Matsumoto}, {\em On {$C$}-reducible {F}insler spaces}, Tensor (N.S.),
  24 (1972), pp.~29--37.
\newblock Commemoration volumes for Prof. Dr. Akitsugu Kawaguchi's seventieth
  birthday, Vol. I.

\bibitem{Mat74}
\leavevmode\vrule height 2pt depth -1.6pt width 23pt, {\em On {F}insler spaces
  with {R}anders' metric and special forms of important tensors}, J. Math.
  Kyoto Univ., 14 (1974), pp.~477--498.

\bibitem{MawWil89}
{\sc J.~Mawhin and M.~Willem}, {\em Critical point theory and {H}amiltonian
  systems}, vol.~74 of Applied Mathematical Sciences, Springer-Verlag, New
  York, 1989.

\bibitem{Mer78}
{\sc F.~Mercuri}, {\em The critical points theory for the closed geodesics
  problem}, Math. Z., 156 (1977), pp.~231--245.

\bibitem{MinSan08}
{\sc E.~Minguzzi and M.~S{\'a}nchez}, {\em The causal hierarchy of spacetimes},
  in Recent developments in pseudo-{R}iemannian geometry, ESI Lect. Math.
  Phys., Eur. Math. Soc., Z\"urich, 2008, pp.~299--358.

\bibitem{Mir04}
{\sc R.~Miron}, {\em The geometry of {I}ngarden spaces}, Rep. Math. Phys., 54
  (2004), pp.~131--147.

\bibitem{O'N83}
{\sc B.~O'Neill}, {\em Semi-{R}iemannian geometry}, vol.~103 of Pure and
  Applied Mathematics, Academic Press Inc. [Harcourt Brace Jovanovich
  Publishers], New York, 1983.
\newblock With applications to relativity.

\bibitem{Per90}
{\sc V.~Perlick}, {\em On {F}ermat's principle in general relativity. {I}.
  {T}he general case.}, Class. Quantum Grav., 7 (1990), pp.~1319--1331.

\bibitem{Per90ii}
{\sc V.~Perlick}, {\em On {F}ermat's principle in general relativity. {II}.\
  {T}he conformally stationary case}, Classical Quantum Gravity, 7 (1990),
  pp.~1849--1867.

\bibitem{Per08}
\leavevmode\vrule height 2pt depth -1.6pt width 23pt, {\em Fermat principle in
  {F}insler spacetimes}, Gen. Relativity Gravitation, 38 (2006), pp.~365--380.

\bibitem{Pham57}
{\sc Q.~Pham}, {\em Inductions {\'e}lectromagn{\'e}tiques en r{\'e}lativit{\'e}
  g{\'e}n{\'e}ral et principe de {F}ermat}, Arch. Ration. Mech. Anal., 1
  (1957), pp.~54--80.

\bibitem{Ran41}
{\sc G.~Randers}, {\em On an asymmetrical metric in the fourspace of general
  relativity}, Phys. Rev. (2), 59 (1941), pp.~195--199.

\bibitem{San97}
{\sc M.~S{\'a}nchez}, {\em Some remarks on causality theory and variational
  methods in {L}orenzian manifolds}, Conf. Semin. Mat. Univ. Bari,  (1997),
  pp.~ii+12.

\bibitem{SSAY77}
{\sc C.~Shibata, H.~Shimada, M.~Azuma, and H.~Yasuda}, {\em On {F}insler spaces
  with {R}anders' metric}, Tensor (N.S.), 31 (1977), pp.~219--226.

\bibitem{syn25}
{\sc J.~Synge}, {\em An alternative treatment of {F}ermat's principle for a
  stationary gravitational field.}, Philos. Mag. and J. of Science, 50 (1925),
  pp.~913--916.

\bibitem{Warner65}
{\sc F.~Warner}, {\em The conjugate locus of a {R}iemannian manifold}, Amer. J.
  Math., 87 (1965), pp.~575--604.

\bibitem{weyl17}
{\sc H.~Weyl}, {\em Zur gravitationstheorie}, Ann. Phys. (Berlin), 54 (1917),
  pp.~117--145.

\bibitem{YaSh77}
{\sc H.~Yasuda and H.~Shimada}, {\em On {R}anders spaces of scalar curvature},
  Rep. Mathematical Phys., 11 (1977), pp.~347--360.

\bibitem{Zau}
{\sc E.~Zaustinsky}, {\em Spaces with non-symmetric distance}, Memoirs of the
  American Mathematical Society, 34 (1959).

\end{thebibliography}
\end{document}